\documentclass[12pt,reqno]{article}

\usepackage[usenames]{color}
\usepackage{amssymb}
\usepackage{graphicx}
\usepackage{amscd}

\usepackage[colorlinks=true,
linkcolor=webgreen,
filecolor=webbrown,
citecolor=webgreen]{hyperref}

\definecolor{webgreen}{rgb}{0,.5,0}
\definecolor{webbrown}{rgb}{.6,0,0}

\usepackage{color}
\usepackage{fullpage}
\usepackage{float}

\usepackage{graphics,amsmath,amssymb}
\usepackage{amsthm}
\usepackage{amsfonts}
\usepackage{latexsym}

\newcommand{\seqnum}[1]{\href{http://oeis.org/#1}{\underline{#1}}}

\begin{document}

\theoremstyle{plain}
\newtheorem{theorem}{Theorem}
\newtheorem{corollary}[theorem]{Corollary}
\newtheorem{lemma}[theorem]{Lemma}
\theoremstyle{remark}
\newtheorem{remark}[theorem]{Remark}

\newcommand{\lrf}[1]{\left\lfloor #1\right\rfloor}

\begin{center}
\vskip 1cm{\LARGE\bf More Fibonacci-Bernoulli relations\\ with and without balancing polynomials \\
\vskip .11in }

\vskip 0.6cm

{\large  Robert Frontczak\footnote{Statements and conclusions made in this article by R. Frontczak are entirely those of the author. They do not necessarily reflect the views of LBBW.
} \\
Landesbank Baden-W\"urttemberg (LBBW) \\ Stuttgart,  Germany \\
\href{mailto:robert.frontczak@lbbw.de}{\tt robert.frontczak@lbbw.de}

\vskip 0.2 in

Taras Goy  \\ 
Faculty of Mathematics and Computer Science\\
Vasyl Stefanyk Precarpathian National University\\
Ivano-Frankivsk, Ukraine\\
\href{mailto:taras.goy@pnu.edu.ua}{\tt taras.goy@pnu.edu.ua}}
\end{center}

\vskip .2 in

\begin{abstract}
We continue our study on relationships between Bernoulli polynomials and balancing (Lucas-balancing) polynomials. From these polynomial relations, we deduce new combinatorial identities with Fibonacci (Lucas) and Bernoulli numbers. Moreover, we prove a special identity involving Bernoulli polynomials and Fibonacci numbers in arithmetic progression.
Special cases and some corollaries will highlight interesting aspects of our findings.
Our results complement and generalize these of Frontczak (2019).
\end{abstract}

\section{Motivation and preliminaries}

Let $B_n(x)$, $x\in\mathbb{C}$, be the $n$-th Bernoulli polynomial defined by
\begin{equation*}
H(x,z) = \sum_{n=0}^\infty B_n(x) \frac{z^n}{n!} = \frac{z e^{xz}}{e^z - 1} \qquad (|z|<2\pi)
\end{equation*}
and $B_n=B_n(0)$ being the $n$-th Bernoulli number \cite{1}.

 Let further $B_n^*(x)$ be the $n$-th balancing polynomial \cite{2}, i.e., polynomials defined by the recurrence
\begin{equation*}
B_n^*(x) = 6x B_{n-1}^*(x) - B_{n-2}^*(x), \qquad n\geq 2,
\end{equation*}
with the initial terms $B_0^*(x)=0$ and $B_1^*(x)=1$. Similarly, Lucas-balancing polynomials are defined by
\begin{equation*}
C_n(x) = 6x C_{n-1}(x) - C_{n-2}(x), \qquad n\geq 2,
\end{equation*}
with the initial terms $C_0(x)=1$ and $C_1(x)=3x$.  See \cite{2,5,AMI,Kim,Meng,Ray-UMJ} for more information about these polynomials. The numbers $B_n^*(1)=B_n^*$ and $C_n(1)=C_n$ are called balancing and Lucas-balancing numbers, respectively (see entries \seqnum{A001109} and \seqnum{A001541} in the On-Line Encyclopedia of Integer Sequences \cite{Sloane}).

Connections between Bernoulli polynomials $B_n(x)$ and balancing polynomials $B_n^*(x)$ are interesting, as they also give relations involving Bernoulli numbers and Fibonacci numbers (we refer to the papers \cite{3,4,7}). The links are
\begin{equation}\label{link1}
B_n^*\Big (\frac{L_{2m}}{6} \Big ) = \frac{F_{2mn}}{F_{2m}}\,, \qquad C_n\Big (\frac{L_{2m}}{6} \Big ) = \frac{L_{2mn}}{2},
\end{equation}
\begin{equation}\label{link2}
B_n^*\Big (\frac{i}{6} L_{2m+1} \Big ) = i^{n-1} \frac{F_{(2m+1)n}}{F_{2m+1}}\,,
\qquad C_n\Big (\frac{i}{6}L_{2m+1} \Big ) = i^n \frac{L_{(2m+1)n}}{2},
\end{equation}
where $m$ is a nonnegative integer, $i=\sqrt{-1}$, and $F_n$ and $L_n$ denote Fibonacci and Lucas numbers, respectively. These sequences are defined by $F_0=0, F_1=1, L_0=2, L_1=1$ and $X_{n}=X_{n-1}+X_{n-2}$ for $n\geq 2$ (entries \seqnum{A000045} and \seqnum{A000032} in \cite{Sloane}).

Recently, Frontczak \cite{3} showed, among other things, that
\begin{equation} \label{fro1}
\sum_{\substack{k=0 \\n \equiv k\, (\!\!\bmod 2)}}^n \!{n\choose k} B^*_{k}(x) (2\sqrt{9x^2-1})^{n-k} B_{n-k} = n C_{n-1}(x).
\end{equation}

Goubi \cite{6} instantly ``improved'' this relation to
\begin{equation} \label{gou1}
\sum_{k=0}^n {n\choose k} B^*_{k}(x) (2\sqrt{9x^2-1})^{n-k} B_{n-k} = n \bigl(C_{n-1}(x) - \sqrt{9x^2-1}B_{n-1}^*(x)\bigr).
\end{equation}
We point out, that since $B_{2n+1}=0$ for $n\geq 1$, the only non-zero contribution in Goubi's sum on the left comes from the index $k=n-1$, which obviously equals
\begin{equation*}
{n\choose n-1} B^*_{n-1}(x) (2\sqrt{9x^2-1}) \Big(-\frac{1}{2}\Big ) = -n \sqrt{9x^2-1}B_{n-1}^*(x).
\end{equation*}
So, the identities \eqref{fro1} and \eqref{gou1} are actually equivalent and the ``improvement'' is a trivial reformulation.
Nevertheless, to keep the notation simple, we will renounce the mod notation and work with the second formulation.

In this paper, we prove more relations between Bernoulli polynomials and balancing polynomials. The proofs are based
on our recent findings concerning exponential generating functions for these polynomials.
From these polynomial relations, we deduce new combinatorial identities with Fibonacci (Lucas) and Bernoulli numbers. Moreover, we prove a special identity involving Bernoulli polynomials
and Fibonacci numbers in arithmetic progression. Some consequences are stated as corollaries.

\section{New Bernoulli-balancing relations}

The next lemma \cite{5} deals with exponential generating functions for balancing and Lucas-balancing polynomials. It will play a key role in the first part of the paper.
\begin{lemma}\label{fundlem1} Let $b_1(x,z)$ and $b_2(x,z)$ be the exponential generating functions of odd and even indexed balancing polynomials, respectively.
Then
\begin{align*} \label{genf_b1}
 b_1(x,z) &= \sum_{n=0}^\infty B_{2n+1}^*(x) \frac{z^n}{n!} \nonumber \\
& = \frac{e^{(18x^2-1)z}}{\sqrt{9x^2-1}}\bigl( 3x \sinh (6x \sqrt{9x^2-1}z) + \sqrt{9x^2-1} \cosh (6x \sqrt{9x^2-1}z)\bigr)
\end{align*}
and
\begin{equation*} \label{genf_b2}
b_2(x,z) = \sum_{n=0}^\infty B_{2n}^*(x) \frac{z^n}{n!} = \frac{e^{(18x^2-1)z}}{\sqrt{9x^2-1}}\sinh (6x \sqrt{9x^2-1}z).
\end{equation*}

Similarly, we have for Lucas-balancing polynomials
\begin{align*} \label{genf_c1}
 c_1(x,z) &= \sum_{n=0}^\infty C_{2n+1}(x) \frac{z^n}{n!} \nonumber \\
& = e^{(18x^2-1)z} \bigl ( 3x \cosh (6x \sqrt{9x^2-1}z) + \sqrt{9x^2-1} \sinh (6x \sqrt{9x^2-1}z)\bigr)
\end{align*}
and
\begin{equation*} \label{genf_c2}
c_2(x,z) = \sum_{n=0}^\infty C_{2n}(x) \frac{z^n}{n!} = e^{(18x^2-1)z}\cosh (6x \sqrt{9x^2-1}z).
\end{equation*}
\end{lemma}

We start with the following results involving even indexed balancing polynomials.
\begin{theorem}\label{thm1}
For each $n\geq 0$ and $x\in\mathbb{C}$, we have
\begin{equation}\label{maineq1}
\sum_{k=0}^n {n\choose k} (12x\sqrt{9x^2-1})^{n-k} B_{n-k}B^*_{2k}(x) = 6x n \bigl(C_{2n-2}(x) - \sqrt{9x^2-1}B_{2n-2}^*(x)\bigr).
\end{equation}
\end{theorem}
\begin{proof}
From
\begin{equation*}
\frac{2}{e^{2x}-1} = \coth x-1
\end{equation*}
we get
\begin{equation*}
H(0,12x\sqrt{9x^2-1}z) = 6xz\sqrt{9x^2-1}\bigl (\coth(6x\sqrt{9x^2-1}z) - 1\bigr).
\end{equation*}

This yields
\begin{align*}
\sum_{n=0}^\infty \Big (\sum_{k=0}^n {n\choose k}&\bigl(12x\sqrt{9x^2-1})^{n-k} B_{n-k}B^*_{2k}(x)\Big)\frac{z^n}{n!} \\
&=  b_2(x,z) H(0,12x\sqrt{9x^2-1}z) \\
& = 6xz e^{(18x^2-1)z}\bigl(\cosh (6x \sqrt{9x^2-1}z) - \sinh (6x \sqrt{9x^2-1}z)\bigr) \\
& = 6xz c_2(x,z) - 6x \sqrt{9x^2-1} z\, b_2(x,z) \\
& = 6x \sum_{n=0}^\infty n \bigl(C_{2n-2}(x) - \sqrt{9x^2-1} B_{2n-2}^*(x)\bigr) \frac{z^n}{n!}.
\end{align*}

The proof is complete.
\end{proof}
\begin{corollary}\label{cor1}
For each $n\geq 0$, the following relation holds
\begin{equation*}\label{eqcor1}
\sum_{k=0}^n {n\choose k}(24\sqrt{2})^{n-k} B_{2k}^*B_{n-k} = 6 n \bigl(C_{2n-2} - 2\sqrt{2}B_{2n-2}^*\bigr).
\end{equation*}
\end{corollary}
\begin{proof}
Set $x=1$ in \eqref{maineq1}.
\end{proof}
\begin{corollary}\label{cor2}
For each $n\geq 0$ and $j\geq 1$,
\begin{equation}\label{eqcor2}
\sum_{k=0}^n {n\choose k}(\sqrt{5}F_{2j})^{n-k}F_{2kj} B_{n-k} = \frac{n}{2}F_{2j} \bigl(L_{2j(n-1)} - \sqrt{5}F_{2j(n-1)}\bigr).
\end{equation}
\end{corollary}
\begin{proof} Evaluate \eqref{maineq1} at the points $x=i/6 L_{2m+1}$ and $x=1/6 L_{2m}$, respectively, and use the links from \eqref{link1} and \eqref{link2}. To simplify the square root recall that $L_n^2=5F_n^2+(-1)^{n}4$.
\end{proof}

The special case
\begin{equation*}
\sum_{k=0}^n {n\choose k} (\sqrt{5})^{n-k} F_{2k}B_{n-k} = \frac{n}{2} \bigl(L_{2n-2} - \sqrt{5}F_{2n-2}\bigr)
\end{equation*}
appears as equation (22) in \cite{3}. We will derive an extension of this result in a sequel.
\begin{remark}
By reindexing, we can write \eqref{maineq1} as follows
\begin{equation}\label{eq1_mod}
\sum_{k=0}^{\lrf{{n}/{2}}}{n\choose 2k}\bigl(144x^2(9x^2-1)\bigr)^kB_{2k}B^*_{2(n-2k)}(x) =6nxC_{2(n-1)}(x).
\end{equation}
\end{remark}

Another interesting identity involving even indexed balancing polynomials is our next theorem.
\begin{theorem} \label{thm2}
For each $n\geq 0$ and $x\in\mathbb{C}$, we have the relation
\begin{gather}
\sum_{k=0}^n {n\choose k} (12x\sqrt{9x^2-1})^{n-k} B^*_{2k}(x)B_{n-k}(x)
\label{maineq2}= 6nx \bigl(18x^2 - 1 + 6x (2x-1)\sqrt{9x^2-1} \bigr)^{n-1}.
\end{gather}
\end{theorem}
\begin{proof}
Since $H(x,z) = \frac{z}{2} \frac{e^{(x-1/2)z}}{\sinh\!\frac{z}{2}}$,
it follows that
\begin{align*}
\sum_{n=0}^\infty \Bigl(\sum_{k=0}^n {n\choose k} (12x\sqrt{9x^2-1})^{n-k}& B^*_{2k}(x) B_{n-k}(x) \Bigr) \frac{z^n}{n!}\\
 & = b_2(x,z)H(x,12x\sqrt{9x^2-1}z) \\
& =  6xz e^{(18x^2-1)z + 6x(2x-1)\sqrt{9x^2-1}z} \\
& =  6x \sum_{n=0}^\infty n \big (18x^2 - 1 + 6x (2x-1)\sqrt{9x^2-1} \big )^{n-1} \frac{z^n}{n!}.
\end{align*}

The proof is complete.
\end{proof}
\begin{corollary}\label{cor3}
For each $n\geq 0$,
\begin{equation}\label{eqcor3}
\sum_{k=0}^n {n\choose k}(3\sqrt{5})^{n-k} \big ( 2^{1-(n-k)}-1\big )  F_{4k} B_{n-k} = 3 n\Big (\frac{7}{2}\Big )^{n-1}.
\end{equation}
\end{corollary}
\begin{proof}
Set $x=1/2$ in \eqref{maineq2} and use that $B_{2n}^*(1/2)\!=\!F_{4n}$ and $B_n(1/2)\!=\!(2^{1-n}-1)B_n$\break  \cite[Corollary 9.1.5]{1}.
\end{proof}

The last identity could be compared with
\begin{equation}\label{id1}
\sum_{k=0}^n {n\choose k} (\sqrt{5})^{n-k} \big ( 2^{1-(n-k)}-1 \big) F_{2k}B_{n-k} = n\Big (\frac{3}{2}\Big )^{n-1},
\end{equation}
which is equation (30) in \cite{3}. It is maybe worth remarking, that the value $x=-1/2$ in conjunction with $B_n^*(-x)=(-1)^{n+1}B_n^*(x)$ \cite{2} and the difference equation for Bernoulli polynomials $B_n(x+1)-B_n(x) =  nx^{n-1}$ \cite[Proposition 9.1.3]{1} gives
\begin{equation*}
\sum_{k=0}^n {n\choose k} (-3\sqrt{5})^{n-k} F_{4k}\big (( 2^{1-(n-k)}-1) B_{n-k}
+ (n-k)(-1)^{n-k}2^{1-(n-k)}\big ) = 3n \Big (\frac{7+6\sqrt{5}}{2}\Big )^{n-1}.
\end{equation*}

So, by Corollary \ref{cor3}, we end with
\begin{equation*}
\sum_{k=0}^n {n\choose k} (n-k) (3\sqrt{5})^{n-k} 2^{1-(n-k)} F_{4k}  = 3n \Bigl(\Big (\frac{7+6\sqrt{5}}{2}\Big )^{n-1}
- \Big (\frac{7}{2}\Big )^{n-1}\Bigr)
\end{equation*}
or, equivalently,
\begin{equation*}
\sum_{k=0}^n {n\choose k} k \Big (\frac{3\sqrt{5}}{2}\Big )^{k}F_{4(n-k)} = \frac{3n}{2^n} \bigl (({7+6\sqrt{5}})^{n-1}
- {7}^{n-1}\bigr).
\end{equation*}
\begin{theorem} \label{thm3}
For each $n\geq 0$ and $x\in\mathbb{C}$, we have the relation
\begin{equation}\label{maineq3}
\sum_{k=0}^{\lrf{{n}/{2}}}{n\choose 2k} (4^k-1)\bigl(144x^2(9x^2-1)\bigr)^k B_{2k} C_{2(n-2k)}(x) =6nx(9x^2-1)B^*_{2(n-1)}(x).
\end{equation}
\end{theorem}
\begin{proof}
Combine $c_1(x,z)$ with $c_2(x,z)$.
\end{proof}
\begin{remark}
Combining $b_2(x,z)$ with $c_1(x,z)$ yields
\begin{equation}
\sum_{k=0}^{\lrf{{n}/{2}}}{n\choose 2k} \bigl(144x^2(9x^2-1)\bigr)^kB_{2k}B^*_{2(n-2k)}(x) =
2n\big (C_{2n-1}(x)-(9x^2-1)B^*_{2(n-1)}(x)\big ).
\end{equation}
Since $C_n(x)=3xC_{n-1}(x)+(9x^2-1)B_{n-1}^*(x)$ \cite[Proposition 2.3]{2}, the right-hand side equals
$2n\big(C_{2n-1}(x)-(9x^2-1)B^*_{2(n-1)}(x)\big)=6nxC_{2(n-1)}(x)$, so we again have \eqref{eq1_mod}. Similarly,
relating $b_1(x,z)$ to $c_2(x,z)$ gives
\begin{gather*}
\sum_{k=0}^{\lrf{{n}/{2}}}{n\choose 2k} (4^k-1)\bigl(144x^2(9x^2-1)\bigr)^kB_{2k}C_{2(n-2k)}(x)\nonumber\\ =2n(9x^2-1)\left(B^*_{2n-1}(x)-C_{2(n-1)}(x)\right),
\end{gather*}
but, since $C_n(x)=B_{n+1}^*(x)-3xB_{n}^*(x)$ \cite[Proposition 2.3]{2}, the right-hand side equals
$2n(9x^2-1)\!\left(B^*_{2n-1}(x)-C_{2(n-1)}(x)\right)=6nx(9x^2-1)B^*_{2(n-1)}(x)$, and we again end with \eqref{maineq3}.
\end{remark}

The next identity is the counterpart of Corollary \ref{cor2}.
\begin{corollary}\label{cor4}
For each $n\geq 0$ and $j\geq 1$,
\begin{equation*}\label{eqcor4}
\sum_{k=0}^{\lrf{{n}/{2}}}{n\choose 2k}(20^{k}-5^{k})F_{2j}^{2k} B_{2k}L_{2j(n-2k)}
= \frac{5n}{2} F_{2j} F_{2j(n-1)}.
\end{equation*}
\end{corollary}
\begin{proof}
Insert $x=i/6 L_{2m+1}$ and $x=1/6 L_{2m}$, respectively, in \eqref{maineq3} to get
\begin{gather*}
\sum_{k=1}^{\lrf{{n}/{2}}}{n\choose 2k}(4^{k}-1) B_{2k}\bigl(L_{2m}^4-4L_{2m}^2\bigr)^k L_{4m(n-2k)}=\frac{n}{2}\frac{L_{2m}}{F_{2m}}(L_{2m}^2-4)F_{4m(n-1)},\\
\sum_{k=1}^{\lrf{{n}/{2}}}{n\choose 2k}(4^{k}-1) B_{2k}\bigl(L_{2m}^4+4L_{2m}^2\bigr)^kL_{2(2m+1)(n-2k)}
=\frac{n}{2}\frac{L_{2m+1}}{F_{2m+1}}(L_{2m+1}^2+4)F_{2(2m+1)(n-1)}.
\end{gather*}

Simplify using $L_n^2=5F_n^2+(-1)^{n}4$ and $L_n F_n = F_{2n}$.
\end{proof}

When $j=1$, then
\begin{equation*}
\sum_{k=0}^{\lrf{{n}/{2}}}{n\choose 2k}(4^{k}-1) 5^k B_{2k}L_{2(n-2k)}=\frac{5n}{2}F_{2(n-1)},
\end{equation*}
which is equation (23) in \cite{3}. When $j=2$, then
\begin{equation*}
\sum_{k=0}^{\lrf{{n}/{2}}}{n\choose 2k}(4^{k}-1) 45^k B_{2k}L_{4(n-2k)}=\frac{15n}{2}F_{4(n-1)}.
\end{equation*}

We conclude the analysis with the following result.
\begin{theorem} \label{thm4}
For each $n\geq 0$ and $x\in\mathbb{C}$, we have the relations
\begin{gather} 
(6x)^{n-1}\sum_{k=0}^{\lrf{{n}/{2}}}{n\choose 2k} (36x^2-4)^k B_{2k}B^*_{n-2k}(x)\nonumber\\
\label{maineq41}= n\left(\sum_{k=0}^{n-1}{n-1\choose k}B^*_{2k+1}(x)-\frac{(6x)^n}{2}B^*_{n-1}(x)\right)
\end{gather}
and
\begin{gather} 
(6x)^{n-1}\sum_{k=0}^{\lrf{{n}/{2}}}{n\choose 2k} (4^k-1)(36x^2-4)^k B_{2k}C_{n-2k}(x)\nonumber\\
\label{maineq42} = n\left(\sum_{k=0}^{n-1}{n-1\choose k}C_{2k+1}(x)-\frac{(6x)^n}{2}C_{n-1}(x)\right).
\end{gather}
\end{theorem}
\begin{proof}
For the first identity, combine $b_1(x,z)$ with $b(x,z)$, where $b(x,z)$ is the exponential generating function for $B_n^*(x)$ \cite{2},	
\begin{equation*}
b(x,z) = \sum_{n=0}^\infty B^*_n(x) \frac{z^n}{n!} = \frac{e^{3xz}}{\sqrt{9x^2-1}}\sinh (\sqrt{9x^2-1}z).
\end{equation*}

The second identity follows from relating $c_1(x,z)$ to $c(x,z)$ with
\begin{equation*}
c(x,z) = \sum_{n=0}^\infty C_n(x) \frac{z^n}{n!} = e^{3xz}\cosh (\sqrt{9x^2-1}z).
\end{equation*}
\end{proof}
\begin{corollary}\label{cor5}
For each $n\geq 0$ and $j\geq 1$,
\begin{gather*}\label{eqcor5}
\sum_{k=0}^{\lrf{{n}/{2}}}{n\choose 2k} (5F_j^2)^{k} F_{j(n-2k)}B_{2k}\\
= (-1)^{nj}\frac{n}{L_{j}^{n-1}}\sum_{k=0}^{n-1}{n-1\choose k} (-1)^{kj} F_{j(2k+1)}-\frac{n}{2}F_{j(n-1)}L_{j}, \\
\sum_{k=0}^{\lrf{{n}/{2}}}{n\choose 2k} (4^{k}-1)(5F_j^2)^{k} L_{j(n-2k)}B_{2k} \nonumber\\
= (-1)^{nj}\frac{n}{L_{j}^{n-1}}\sum_{k=0}^{n-1}{n-1\choose k} (-1)^{kj} L_{j(2k+1)}-\frac{n}{2}L_{j(n-1)}L_{j}.
\end{gather*}
\end{corollary}
\begin{proof}
Set $x=i/6 L_{2m+1}$ and $x=1/6 L_{2m}$ in \eqref{maineq41} and \eqref{maineq42}, respectively, and simplify as before.
\end{proof}

\section{A special polynomial identity}

Equations \eqref{eqcor2}, \eqref{eqcor3} and \eqref{id1} give rise to the question, if there is a connection between them. The answer to that question is positive, as will be shown in the next theorem. The theorem generalizes Theorem 9 in \cite{3}, which has been generalized in a different way in \cite{4}.
The proof of the extension presented here does not require the notion of balancing polynomials.
\begin{theorem} \label{thm5}
Let $\alpha$ be the golden ratio, $\alpha=(1+\sqrt{5})/2$, and $\beta=(1-\sqrt{5})/2=-1/\alpha$. Then, for each $n\geq 0$, $j\geq 1$, and $x\in\mathbb{C}$, we have the relations
\begin{equation}\label{main51}
\sum_{k=0}^n {n\choose k} F_{jk} (\sqrt{5}F_j)^{n-k} B_{n-k}(x) = n F_j \big ((\sqrt{5}x+\beta)F_j + F_{j-1}\big )^{n-1}
\end{equation}
and
\begin{equation}\label{main52}
\sum_{k=0}^n {n\choose k} F_{jk} (-\sqrt{5}F_j)^{n-k} B_{n-k}(x) = n F_j \big ((\alpha - \sqrt{5}x)F_j + F_{j-1}\big )^{n-1}.
\end{equation}
\end{theorem}
\begin{proof}
Let $F(z)$ be the exponential generating function for $(F_{jn})_{n\geq 0},j\geq 1$. Evidently, the Binet formula for $F_{jn}$ gives
\begin{equation*}
F(z) = \sum_{n=0}^\infty F_{jn} \frac{z^n}{n!} = \frac{1}{\sqrt{5}}\big (e^{\alpha^j z} - e^{\beta^j z}\big ).
\end{equation*}

Now we use the relations $\alpha^j = \alpha F_j + F_{j-1}$ and $\beta^j = \beta F_j + F_{j-1}$,
to write
\begin{equation*}
F(z) = \frac{2}{\sqrt{5}} e^{(1/2 F_j + F_{j-1})z}\sinh\Big ( \frac{\sqrt{5}F_j}{2}  z\Big ).
\end{equation*}

Hence, it follows that
\begin{align*}
\sum_{n=0}^\infty \Big (\sum_{k=0}^n {n\choose k} F_{jk} (\sqrt{5}F_j)^{n-k} B_{n-k}(x) \Big ) \frac{z^n}{n!} & =
F(z) H(x,\sqrt{5}F_j z) \\
& =  F_j z\, e^{((x-1/2)\sqrt{5}F_j + 1/2 F_j + F_{j-1})z} \\
& =  F_j z\, e^{((\sqrt{5}x + \beta)F_j + F_{j-1})z}.
\end{align*}

This proves the first equation. The second follows upon replacing $x$ by $1-x$
and using $B_n(1-x)=(-1)^nB_n(x)$ \cite[Proposition 9.1.3]{1}  and $\alpha-\beta=\sqrt{5}$.
\end{proof}

When $x=0$, then
\begin{equation*}
\sum_{k=0}^n {n\choose k}  (\sqrt{5}F_j)^{n-k} F_{jk}B_{n-k} = n F_j \beta^{j(n-1)}
= \frac{n}{2}F_{j} \big (L_{j(n-1)} - \sqrt{5}F_{j(n-1)}\big ),
\end{equation*}
which generalizes \eqref{eqcor2}. Similarly,
\begin{equation*}
\sum_{k=0}^n {n\choose k}  (-\sqrt{5}F_j)^{n-k} F_{jk}B_{n-k} = n F_j \alpha^{j(n-1)}
= \frac{n}{2}F_{j} \big (L_{j(n-1)} + \sqrt{5}F_{j(n-1)}\big ).
\end{equation*}

A combination of both yields
\begin{equation*}
\sum_{k=0}^n {n\choose k}  (\sqrt{5}F_j)^{n-k}\bigl(1+(-1)^{n-k}\bigr) F_{jk}B_{n-k} = n F_j L_{j(n-1)}.
\end{equation*}
\begin{corollary}\label{cor6} For each $n\geq 0$ and $j\geq 1$,
\begin{equation}\label{eqcor6}
\sum_{k=0}^n {n\choose k} (\sqrt{5}F_j)^{n-k} \big ( 2^{1-(n-k)}-1\big ) F_{jk}B_{n-k} = n 2^{1-n} F_j L_j^{n-1}.
\end{equation}
\end{corollary}
\begin{proof}
Set $x=1/2$ in \eqref{main51} or \eqref{main52} and use once more $B_n(1/2)=(2^{1-n}-1) B_n$. When simplifying, keep in mind the relation $F_j+2F_{j-1}=L_j$.
\end{proof}

When $j=2$ and $j=4$, respectively, we get \eqref{eqcor3} and \eqref{id1}. For $j=3$, the identity becomes
\begin{equation*}
\sum_{k=0}^n {n\choose k} (2\sqrt{5})^{n-k} \big ( 2^{1-(n-k)}-1\big ) F_{3k}B_{n-k} = n 2^{n}.
\end{equation*}
\begin{corollary}\label{cor7}
Let $n$, $j$ and $q$ be integers with $n$, $j\geq 1$ and $q\geq 2$. Then it holds that
\begin{equation} \label{eqcor7}
\sum_{k=0}^{n} {n\choose k}(\sqrt{5}F_j )^{n-k} \big ( q^{1-(n-k)} - 1 \big )F_{jk} B_{n-k} =
n F_j q^{1-n} \sum_{r=1}^{q-1} \big ( r\alpha^j + (q - r)\beta^j \big )^{n-1}.
\end{equation}
\end{corollary}
\begin{proof} The multiplication theorem \cite[Proposition 9.1.3]{1} $\frac{1}{q} \sum\limits_{r=0}^{q-1} B_n \Big (x + \frac{r}{q}\Big ) = \frac{B_n(qx)}{q^{n}}
$
gives
\begin{equation*}
\big ( q^{1-(n-k)} - 1 \big ) B_{n-k} = \sum_{r=1}^{q-1} B_n \Big (\frac{r}{q}\Big ).
\end{equation*}

Therefore, we can write
\begin{align*}
\sum_{k=0}^{n} {n\choose k} F_{jk} (\sqrt{5}F_j)^{n-k} \big ( q^{1-(n-k)} - 1 \big ) B_{n-k} & = n F_j \sum_{r=1}^{q-1} \Big ( \Big ( \sqrt{5}\,\frac{r}{q}+\beta\Big ) F_j + F_{j-1}\Big )^{n-1} \\
& = n F_j q^{1-n} \sum_{r=1}^{q-1} \big ( \sqrt{5}r F_j + q(\beta F_j + F_{j-1})\big )^{n-1} \\
& = n F_j q^{1-n} \sum_{r=1}^{q-1} \big ( r\alpha^j + (q - r)\beta^j \big )^{n-1}.
\end{align*}
\end{proof}

When $q=2$, then \eqref{eqcor7} gives \eqref{eqcor6}. When $q=3$, then we obtain
\begin{equation*}
\sum_{k=0}^{n} {n\choose k} F_{jk} (\sqrt{5}F_j )^{n-k} \big ( 3^{1-(n-k)} - 1 \big ) B_{n-k}
= n F_j 3^{1-n}\big ( ( L_{j}+\beta^j)^{n-1} + (L_j+\alpha^j)^{n-1}\big).
\end{equation*}
\begin{corollary}\label{cor8}
Let $n$, $j$ and $m$ be integers with $n$, $j\geq 1$ and $0\leq m \leq n-1$. Then
\begin{gather*} \label{eqcor81}
\sum_{k=0}^{n-m} {n\choose k} F_{jk} (\sqrt{5}F_j )^{n-k} (n-k)_m B_{n-m-k}(x) \nonumber\\
=
(n)_{m+1} F_j  (\sqrt{5}F_j )^{m} \big ((\sqrt{5}x+\beta)F_j + F_{j-1}\big )^{n-1-m}
\end{gather*}
and
\begin{gather*} \label{eqcor82}
\sum_{k=0}^{n-m} {n\choose k} F_{jk} (-\sqrt{5}F_j)^{n-k} (n-k)_m B_{n-m-k}(x)\nonumber \\=
(n)_{m+1} F_j \big (-\sqrt{5}F_j \big )^{m} \big ((\alpha - \sqrt{5}x)F_j + F_{j-1}\big )^{n-1-m},
\end{gather*}
where $(y)_n=y(y-1)\cdots (y-n+1)$, $(y)_0=1$, denotes the falling factorial.
\end{corollary}
\begin{proof}
Differentiate the identities in Theorem \ref{thm5} $m$ times and use the fact $B_n'(x)=n B_{n-1}(x)$ \cite[Proposition 9.1.2 ]{1}. When $m\geq n$, then both sides of the identities become zero.
\end{proof}
\begin{corollary}\label{cor9}
For nonnegative integers $n$, $N$ and $j\geq 1$, we have the identities
\begin{equation*} \label{eqcor91}
\sum_{s=0}^{N} \Big (\big (\alpha^j+\sqrt{5}F_j s \big)^n - \big (\beta^j+\sqrt{5}F_j s \big)^n\Big )
= \big (\sqrt{5}F_j N + \alpha^j\big )^n - \beta^{jn}
\end{equation*}
and
\begin{equation*} \label{eqcor92}
\sum_{s=0}^{N} \Big (\big (\beta^j-\sqrt{5}F_j s \big)^n - \big (\alpha^j-\sqrt{5}F_j s \big)^n\Big )
= \big (-\sqrt{5}F_j N + \beta^j\big )^n - \alpha^{jn}.
\end{equation*}
\end{corollary}
\begin{proof} We only prove the first identity. Integrate both sides of \eqref{main51} from $0$ to $N+1$ and use the formula
\begin{displaymath}
\sum_{s=0}^{N} s^n = \int_{0}^{N+1}B_{n}(x)dx.
\end{displaymath}
The last integral identity actually reads as
\begin{displaymath}
\sum_{s=0}^{N} s^n = \int_{0}^{N+1}B_{n}(x) dx = \frac{1}{N+1} \big (B_{n+1}(N+1)-B_{n+1}\big )
\end{displaymath}
(Faulhaber's formula) and holds for all $n\geq 2$, so a justification is needed. As we will work with the integral part only, with the convention that $0^0=1$, the cases $n=0$ and $n=1$ can be checked explicitly. Hence, for the LHS of \eqref{main51} we obtain
\begin{eqnarray*}
\sum_{k=0}^{n} \binom{n}{k} F_{jk} \big (\sqrt{5}F_j \big )^{n-k} \int_{0}^{N+1}B_{n-k}(x)dx & = &
\sum_{s=0}^N\sum_{k=0}^{n} \binom{n}{k} F_{jk} \big (\sqrt{5}F_j s \big )^{n-k} \\
& = & \frac{1}{\sqrt{5}} \sum_{s=0}^N \Big (\big (\alpha^j+\sqrt{5}F_j s \big)^n - \big (\beta^j+\sqrt{5}F_j s \big)^n\Big ).
\end{eqnarray*}
The integral on the RHS of \eqref{main51} is easily evaluated as
\begin{displaymath}
n F_j\int_{0}^{N+1}\big((\sqrt{5}x+\beta)F_j + F_{j-1}\big )^{n-1}dx = \frac{1}{\sqrt{5}}\big (\sqrt{5}F_j (N+1) + \beta^j\big )^n - \beta^{jn}.
\end{displaymath}
The proof of the second formula is similar.
\end{proof}

\section{Conclusion}
In this article, we have discovered new identities relating Bernoulli numbers (polynomials) to balancing and Lucas-balancing polynomials.
We have also derived a general identity involving Bernoulli polynomials and Fibonacci numbers in arithmetic progression.
In our future papers, we will discuss the analogue results for Euler polynomials and Lucas-balancing polynomials as well as
identities connecting Bernoulli polynomials with Fibonacci and Chebyshev polynomials.

\bigskip
\hrule
\bigskip

\noindent 2010 {\it Mathematics Subject Classification}: 11B37, 11B65, 05A15.

\bigskip
\noindent \emph{Keywords:} Bernoulli polynomials and numbers, Balancing polynomials and numbers, \break Fibonacci numbers, Lucas numbers, generating function.

\end{document}